\documentclass{amsart}

\usepackage{graphicx}%
\usepackage{multirow}%
\usepackage{amsmath,amssymb,amsfonts}%
\usepackage{amsthm}%
\usepackage{mathrsfs}%
\usepackage[title]{appendix}%
\usepackage{xcolor}%
\usepackage{textcomp}%
\usepackage{manyfoot}%
\usepackage{booktabs}%
\usepackage{algorithm}%
\usepackage{algorithmicx}%
\usepackage{algpseudocode}%
\usepackage{listings}%

\usepackage{latexsym}
\usepackage[inline]{enumitem} 
\usepackage{circledsteps}
\usepackage{marvosym}
\allowdisplaybreaks 

 \usepackage{comment}

\newtheorem{theorem}{Theorem}[section]
\newtheorem{proposition}[theorem]{Proposition}

\newtheorem{remark}{Remark}[section]
\newtheorem{definition}{Definition}[section]

\newcommand{\supp}{\operatorname{supp}}
\newcommand{\singsupp}{\operatorname{sing\, supp}}

\newtheorem{lemma}[theorem]{Lemma}
\theoremstyle{definition}

\begin{document}

\title{On the continuity of the product of distributions in local Sobolev spaces}

\author{Stefan Tuti\'c}
\address{Department of Mathematics and Informatics \\ 
Faculty of Sciences University of Novi Sad \\
Trg Dositeja Obradovi\'ca 3 \\
21000 Novi Sad, Serbia}
\email{stefan.tutic@dmi.uns.ac.rs}

\begin{abstract} 
We consider the space $\mathscr{H}_L ^{s,r} (O)$ consisting of all local Sobolev distributions of order $s$ on an open set $O$ whose Sobolev wave front set of order $r$ is contained in the closed conic set $L\subseteq O\times(\mathbb{R}^m\backslash\{0\})$. We introduce a locally convex topology on $\mathscr{H}_L ^{s,r} (O)$ and show that the ordinary product of smooth functions uniquely extends to a continuous bilinear mapping $\mathscr{H}_{L_1} ^{r_1,r'} (O) \times \mathscr{H}_{L_2} ^{r_2,r''} (O) \to \mathscr{H}_{L} ^{s,r} (O)$, for appropriate $s$ and $r$ when $L_1$ and $L_2$ are in a favorable position. The key ingredient in our proof is to employ H\"ormander's idea of considering the pullback by the diagonal map $x\mapsto (x,x)$ of the tensor product of two distributions.
\end{abstract}


\maketitle

\section{Introduction}\label{sec1}

Multiplication of two distributions is one of the fundamental operations in quantum field theory and the theory of partial differential equations. It is well known that their product cannot be well defined in the general case (see \cite{Schwartz}). For some distributions, it is possible to define their product by using wave front sets, which were introduced by H\"ormander (for more details see \cite{HNonlinear}). In \cite{Duistermaat}, Duistermaat and H\"ormander introduced the Sobolev wave front set to study the propagation of singularities of solutions of partial differential equations.

In this paper, we consider Sobolev globally regular distributions of order $s$, defined on the open subset $O$ of $\mathbb{R}^m$, that are microlocally Sobolev regular of order $r$, denoted by $\mathscr{H}_L ^{s,r} (O)$, where the microlocal property is described by the Sobolev wave front set which is contained in a given closed conic subset $L$ of $O\times (\mathbb{R}^m \backslash \{ 0 \})$. The main goal of this paper is to show that the product of two distributions from $\mathscr{H}_{L_1} ^{r',r_1} (O) \times \mathscr{H}_{L_2} ^{r'',r_2} (O)$ is again a distribution of this kind $\mathscr{H}_{L} ^{s,r} (O)$, where $s,r,L$ depend on the initial parameters $r',r_1,L_1,r'',r_2,L_2$. This is obtained by using standard techniques of the pullback by the diagonal mapping of the tensor product of distributions. Moreover, we equip $\mathscr{H}_L ^{s,r} (O)$ with a locally convex topology in a natural way and show that the product we define is a continuous bilinear mapping $\mathscr{H}_{L_1} ^{r_1,r'} (O) \times \mathscr{H}_{L_2} ^{r_2,r''} (O) \to \mathscr{H}_{L} ^{s,r} (O)$.

The motivation for our work comes from the recent results obtained in \cite{Brouder} and \cite{PP}. In \cite{Brouder}, the authors studied the continuity of the product of two distributions using the classical $C^{\infty}$ wave front set. Originally, this was done by H\"ormander, but his topology was weak, in the sense that the pullback mapping was not continuous; it is only sequentially continuous. The authors in \cite{Brouder}, modified the topology so that the pullback mapping is continuous (see also \cite{dab-br}). The ideas they used in their work can be implemented in our case. In \cite{PP}, the authors introduced a locally convex topology on the space $\mathcal{D}_{L}^{\,\prime\, r} (O)$, consisting of the distributions whose Sobolev wave front set of order $r$ is contained in a given closed conic set $L$. For this topology, the pullback mapping is continuous for a certain order of Sobolev microregularity. Employing this notations, we set $\mathscr{H}_L ^{s,r} (O):=H_{loc} ^{s} (O)\cap \mathcal{D}_{L}^{\,\prime\, r} (O)$ and $\mathscr{H}_L ^{s,r} (O)$ will carry the topology induced by the continuous seminorms on $\mathcal{D}_{L}^{\,\prime\, r} (O)$ together with the continuous seminorms on $H_{loc} ^{s} (O)$.

\section{Preliminaries}\label{sec2}

First, we fix the notations we are going to employ throughout the paper. The sets of positive integers, nonnegative integers, and real numbers are denoted by $\mathbb{N}$, $\mathbb{N}_0$ and $\mathbb{R}$, respectively. A subset $V$ of $\mathbb{R}^n$ is a cone if $\lambda x\in V$ for every $\lambda>0,x\in V$. On the other hand, a subset $L$ of $\dot{T}^*O = O\times (\mathbb{R}^n\backslash\{0\})$ is conic if it is a cone in the second variable. The Fourier transform is defined as $\widehat{f} (\xi)=\mathcal{F}f(\xi) := \int_{\mathbb{R}^n} f(x) e^{-ix\xi} \mathrm{d}x$ for $f\in L^1 (\mathbb{R}^n)$, and the Japanese bracket is defined by $\langle x \rangle := (1+|x|^2)^{\frac{1}{2}}$ for all $x\in \mathbb{R}^n$. With $H^s (\mathbb{R}^n) = \{ u\in\mathcal{S}'(\mathbb{R}^n) \mid \langle \cdot \rangle^s \mathcal{F}u \in L^2 (\mathbb{R}^n) \}$ we denote the Sobolev space of order $s\in\mathbb{R}$. The local Sobolev space of order $s\in\mathbb{R}$ defined on the open subset $O$ of $\mathbb{R}^n$ is denoted by
\begin{align}
\label{Lokalni_Soboljev}
H_{loc} ^s (O) = \{ u\in\mathcal{D}' (O) \mid \varphi u \in H^{s} (\mathbb{R}^n) \text{ for all } \varphi\in\mathcal{D} (O) \}.
\end{align}
For more details on Sobolev spaces, we refer to \cite{Adams}.

The Sobolev wave front set $WF^r (u)$ of order $r\in\mathbb{R}$ of $u\in\mathcal{D}'(O)$ was introduced in \cite{Duistermaat} and it is defined as follows.
For $r\in\mathbb{R}$ and $v\in \mathcal{E}'(O)$, set
\begin{align*}
    \Sigma^r (v) := \{ \xi\in\mathbb{R}^n \backslash \{ 0 \} \mid \neg (\text{ exists an open cone $V\ni\xi$ such that } \|\langle\cdot\rangle^r\mathcal{F}(v)\|_{L^2(V)}<\infty) \}.
\end{align*}
For  $u\in \mathcal{D}'(O)$,
\begin{align*}
    \Sigma_x ^r (u)
    &:= \bigcap_{ \substack{\varphi\in \mathcal{D} (O) \\ \varphi (x) \neq 0}} \Sigma^r (\varphi u), \quad x\in O, \\
    WF^r (u) &:= \{ (x,\xi)\in O \times (\mathbb{R}^n \backslash \{ 0 \}) \mid \xi \in \Sigma_x ^r (u) \}.
\end{align*}
Set $\mathcal{D}_{L}^{\,\prime\, r} (O) := \{ u\in \mathcal{D}' (O) \mid  WF^r (u) \subseteq L\}$, where $L$ is a closed conic subset of $O \times (\mathbb{R}^n \backslash \{ 0 \})$. Following \cite{PP}, we equip $\mathcal{D}_{L}^{\,\prime\, r} (O)$ with the locally convex topology induced by all seminorms $p_B (u) := \sup_{\varphi\in B} |\langle u,\varphi \rangle|$, where $B$ is a bounded subset of $\mathcal{D}(O)$, together with all seminroms
\begin{align*}
p_{r;\varphi , V} (u): = \left( \int_V \langle \xi \rangle^r |\mathcal{F} (\varphi u) (\xi)|^2 \mathrm{d} \xi \right)^{\frac{1}{2}}
\end{align*}
where $\varphi\in\mathcal{D} (O)$ and the closed cone $V\subseteq\mathbb{R}^n$ satisfy $(\supp \varphi \times V)\cap L = \emptyset$. We recall the following result from \cite{PP}.

\begin{theorem}[{\cite[Theorem 3.21]{PP}}]
\label{PPteorema}
    Let $O$ and $U$ be open subsets of $\mathbb{R}^m$ and $\mathbb{R}^n$, respectively, let $f:O \to U$ be a smooth map and let $L$ be a closed conic subset of $U\times (\mathbb{R}^n \backslash \{ 0 \})$ satisfying $L\cap \mathcal{N}_f = \emptyset$, where
    \begin{align*}
        \mathcal{N}_f = \{ (f(x), \eta) \in U \times \mathbb{R}^n \mid x\in O, \ ^t f'(x) \eta = 0 \}.
    \end{align*}
    The pullback $f^*: \mathcal{C}^{\infty}(U) \to \mathcal{C}^{\infty}(O), f^*(u)=u\circ f$, uniquely extends to a well defined and continuous mapping $f^*: \mathcal{D}_{L}^{\,\prime\, r_2}  (U) \to \mathcal{D}_{f^* L}^{\,\prime\, r_1} (O)$ when $r_2 - r_1>\frac{n}{2} , r_2 > \frac{n}{2}$, where
    \begin{align}
    \label{pullback_skupa}
        f^*L = \{ (x, \, ^tf'(x)\eta)\in O \times\mathbb{R}^m \mid (f(x),\eta)\in L \}.
    \end{align}
    If $f$ has a constant rank $k\geq 1$, then this is valid when $r_2-r_1\geq \frac{n-k}{2}$ and $r_2 > \frac{n-k}{2}$. When $f$ is a submersion, $f^*: \mathcal{D}_{L}^{\,\prime\, r_2}  (U) \to \mathcal{D}_{f^* L}^{\,\prime\, r_1} (O)$ is well-defined and continuous even when $r_2\geq r_1$. Consequently, if $f$ is a diffeomorphism, then $f^*: \mathcal{D}_{L}^{\,\prime\, r}  (U) \to \mathcal{D}_{f^* L}^{\,\prime\, r} (O)$ is a topological isomorphism for each $r\in\mathbb{R}$.
\end{theorem}

\section{The main results}\label{sec3} 

We start with the following technical result.

\begin{lemma}
\label{WF_tenz_proiz}
Let $O\subseteq \mathbb{R}^m$ and $U\subseteq \mathbb{R}^n$ be open sets. Let $r,r_1,r_2,r',r''\in\mathbb{R}$ satisfy $r \leq r_1 + \min \{0,r''\}$ and $r \leq r_2 + \min \{0,r'\}$. For all $u\in H_{loc} ^{r'} (O), v\in H_{loc} ^{r''} (U)$ it holds that
\begin{align}
\label{talsnifrontinkl}
\begin{split}
WF^r (u\otimes v)
&\subseteq (WF^{r_1} (u) \times WF^{r_2} (v)) \cup (WF^{r_1} (u) \times (\supp v \times \{0\})) \\
&\quad \cup ((\supp u \times \{0 \}) \times WF^{r_2} (v)).
\end{split}
\end{align}
\end{lemma}
\begin{proof}
We start off by showing
\begin{align}
\label{sigma}
\Sigma^r (u\otimes v)
\subseteq
(\Sigma^{r_1} (u) \times \Sigma^{r_2} (v)) \cup (\Sigma^{r_1} (u) \times \{0\}) \cup (\{0\} \times \Sigma^{r_2} (v)),
\end{align}
for $u\in \mathcal{E}' (O) \cap H_{loc} ^{r'} (O) , v\in \mathcal{E}' (U) \cap H_{loc} ^{r''} (U)$. Let $(\xi_0,\eta_0)\notin (\Sigma^{r_1} (u) \times \Sigma^{r_2} (v)) \cup (\Sigma^{r_1} (u) \times \{0\}) \cup (\{0\} \times \Sigma^{r_2} (v))$ satisfies $(\xi_0,\eta_0)\neq (0,0)$. We distinguish three cases:
\begin{enumerate}
\item \underline{$\xi_0 = 0$:}

From $(\xi_0,\eta_0)\notin \{0\} \times \Sigma^{r_2} (v)$, that is $\eta_0 \notin \Sigma^{r_2} (v)$, we get there exists an open cone $V$ which is a neighborhood of $\eta_0$ such that
\begin{align}
\label{eta0}
\lVert \langle\cdot \rangle^{r_2} \mathcal{F} v \rVert_{L^2 (V)}
< \infty.
\end{align}
Set $W:=\{(\xi , \eta) \in \mathbb{R}^{m+n} \mid \eta\in V , |\xi| < |\eta| \}$. Clearly, $(\xi_0,\eta_0)\in W$. If $r'\geq 0$ then $\hat{u}\in L^2 (\mathbb{R}^m)$, because $u\in H_{loc}  ^{r'} (O)$ and has a compact support. Therefore,
\begin{align*}
\lVert \langle \cdot \rangle^{r} \mathcal{F} (u\otimes v)  \rVert_{L^2 (W)} ^2
&= \iint_{\substack{ \eta\in V \\ |\xi|<|\eta| }} \langle (\xi ,\eta) \rangle^{2r} |\mathcal{F}u (\xi)|^2 |\mathcal{F}v (\eta)|^2 \mathrm{d}\xi \mathrm{d}\eta \\
\boxed{\substack{\langle (\xi ,\eta) \rangle^{2r} \leq c \langle \eta \rangle^{2r}, \\ \text{for } |\xi| < |\eta| }}
&\leq c \iint_{\substack{ \eta\in V \\ |\xi|<|\eta| }} \langle \eta \rangle^{2r} |\mathcal{F}u (\xi)|^2 |\mathcal{F}v (\eta)|^2 \mathrm{d}\xi \mathrm{d}\eta\\
\boxed{\substack{ r \leq r_2 + \min \{0,r'\} = r_2 \\ }}
&\leq c \iint_{\substack{ \eta\in V \\ |\xi|<|\eta| }} \langle \eta \rangle^{2r_2} |\mathcal{F}u (\xi)|^2 |\mathcal{F}v (\eta)|^2 \mathrm{d}\xi \mathrm{d}\eta\\
&\leq c \iint_{\substack{ \eta\in V \\ \xi \in \mathbb{R}^m }} \langle \eta \rangle^{2r_2} |\mathcal{F}u (\xi)|^2 |\mathcal{F}v (\eta)|^2 \mathrm{d}\xi \mathrm{d}\eta\\
&= c\int_{\mathbb{R}^m} |\mathcal{F}u (\xi)|^2 \mathrm{d}\xi \int_{\eta\in V} \langle \eta \rangle^{2r_2} |\mathcal{F}v (\eta)|^2 \mathrm{d}\eta
<\infty.
\end{align*}
Now, let $r'<0$. When $r_2 \geq 0 $, and $|\xi|\leq |\eta|$, it follows
\begin{align}
\label{triv}
\langle ( \xi , \eta ) \rangle^r
\leq \langle ( \xi , \eta ) \rangle^{r_2} \langle ( \xi , \eta ) \rangle^{r'}
\leq c \langle \eta \rangle^{r_2} \langle \xi \rangle^{r'}.
\end{align}
For $r_2 < 0$, \eqref{triv} trivially holds. Therefore,
\begin{align*}
\lVert \langle \cdot \rangle^{r} \mathcal{F} (u\otimes v) \rVert_{L^2 (W)} ^2
&\leq c^2 \iint_{\substack{ \eta\in V \\ |\xi|<|\eta| }} \langle \eta \rangle^{2r_2} \langle \xi \rangle^{2r'} |\mathcal{F}u (\xi)|^2 |\mathcal{F}v (\eta)|^2 \mathrm{d}\xi \mathrm{d}\eta\\
&\leq c^2\int_{\xi\in\mathbb{R}^m} \langle \xi \rangle^{2r'} |\mathcal{F}u (\xi)|^2 \mathrm{d}\xi \int_{\eta\in V } \langle \eta \rangle^{2r_2} |\mathcal{F}v (\eta)|^2 \mathrm{d}\eta
< \infty,
\end{align*}
since $u\in H_{loc}  ^{r'} (O)$ and has a compact support.

\item \underline{$\eta_0 = 0$:}

The proof is analogous to the previous case.

\item \underline{$\xi_0\neq0 , \eta_0\neq 0 $:}

Since $(\xi_0, \eta_0)\notin \Sigma^{r_1} (u) \times \Sigma^{r_2} (v)$, lets assume that $\eta_0 \notin \Sigma^{r_2} (v)$. This means that \eqref{eta0} holds true. Choose $R>\frac{|\xi_0|}{|\eta_0|}$ and notice that $(\xi_0 , \eta_0) \in \{ (\xi , \eta)\in \mathbb{R}^{m+n} \mid \eta\in V , |\xi| < R |\eta | \} =: W$. Now, arguing in the same manner as in the first case, we get that
\begin{align*}
\lVert \langle \cdot\rangle^{r} \mathcal{F} (u\otimes v) \rVert_{L^2 (W)} < \infty.
\end{align*}
The case $\xi_0 \notin \Sigma^{r_1} (u)$ follows analogously.
\end{enumerate}
For $u\in \mathcal{D}' (O), v\in \mathcal{D}' (U), (x_0 , y_0) \in O\times U$ and $s\in\mathbb{R}$ it holds that
\begin{align}
\label{sigmaxy}
\begin{split}
&\Sigma_{(x_0 , y_0)} ^s (u\otimes v)\\
&= \bigcap \{ \Sigma^s ( (\varphi u) \otimes (\psi v) ) \mid \varphi \in \mathcal{D} (O) , \varphi (x_0) \neq 0 , \psi\in\mathcal{D} (U) , \psi (y_0) \neq 0\}.
\end{split}
\end{align}
This follows from the fact that $\Sigma^s (\rho w) \subseteq \Sigma^s (w)$ for $\rho \in \mathcal{D} (O) , w \in \mathcal{E}' (O)$. To see this, assume that $(\xi_0 , \eta_0)$ belongs to the right hand side of (\ref{sigmaxy}). Then for arbitrary $\phi\in\mathcal{D} (O\times U) , \phi (x_0 , y_0) \neq 0$ we can choose open sets $M\subseteq O , N\subseteq U$ such that $(x_0,y_0)\in M\times N$ and $\phi \neq 0$ on $M\times N$. Pick $\varphi\in \mathcal{D} (M)$ and $\psi\in\mathcal{D}(N)$ such that $\varphi (x_0) \neq 0$ and $\psi (y_0) \neq 0$. As $\varphi \otimes \psi = (\varphi \otimes \psi)\phi/\phi$, we infer $(\xi_0 , \eta_0) \in \Sigma^s ((\varphi u) \otimes (\psi v)) = \Sigma^s (\phi ((\varphi \otimes \psi)/\phi) (u\otimes v)) \subseteq \Sigma^s (\phi (u\otimes v)) $. The other inclusion follows immediately from the definition of the set $\Sigma_{(x_0 , y_0)} ^s (u\otimes v)$.

Now, let $u\in H_{loc} ^{r'} (O), v\in H_{loc} ^{r''} (U)$. From (\ref{sigma}) and (\ref{sigmaxy}) we infer
\begin{align*}
\Sigma_{(x_0 , y_0)} ^r (u\otimes v)
\subseteq (\Sigma_{x_0 } ^{r_1} (u) \times \Sigma_{y_0 } ^{r_2} (v)) \cup (\Sigma_{x_0 } ^{r_1} (u) \times \{0 \}) \cup ( \{0 \} \times \Sigma_{y_0 } ^{r_2} (v)),
\end{align*}
and therefore
\begin{align*}
WF^r (u\otimes v)
&= \bigcup_{\substack{ x\in\supp u \\ y\in \supp v }} \{(x,y)\} \times \Sigma_{(x,y)} ^r (u\otimes v)\\
&\subseteq (WF^{r_1} (u) \times WF^{r_2} (v)) \cup (WF^{r_1} (u) \times (\supp v \times \{0\}))\\
&\quad \cup ((\supp u \times \{0 \}) \times WF^{r_2} (v)).
\end{align*}
\end{proof}

Since in the previous lemma we assumed a certain global Sobolev regularity, we introduce the following space.

\begin{definition}
\label{definicija_prostora}
Let $s,r\in\mathbb{R}$. Let $O\subseteq\mathbb{R}^m$ be an open set and $L$ a closed conic subset of $O\times (\mathbb{R}^m \backslash \{0\})$. We define
\begin{align*}
\mathscr{H}_L ^{s,r} (O) := \{ u\in H_{loc} ^s (O) \mid WF^r (u) \subseteq L \}=H_{loc} ^s (O)\cap \mathcal{D}_{L}^{\,\prime\, r} (O).
\end{align*}
\end{definition}
The index $s$ represents global Sobolev regularity of the distribution $u$ while $r$ represents microlocal Sobolev regularity of $u$. We equip $\mathscr{H}_L ^{s,r} (O)$ with the locally convex topology induced by the family of seminorms:
\begin{align}
\label{seminorme}
\begin{split}
p_{r;\varphi,V} (u)
&= \left( \int_V \langle \xi \rangle^{2r} |\mathcal{F} (\varphi u) (\xi)|^2 \mathrm{d}\xi \right)^{\frac{1}{2}}, \quad
q_{s;\psi} (u)
= \left( \int_{\mathbb{R}^m} \langle \xi \rangle^{2s} |\mathcal{F} (\psi u) (\xi)|^2 \mathrm{d}\xi \right)^{\frac{1}{2}},
\end{split}
\end{align}
where $\psi\in\mathcal{D}(O)$ and $\varphi\in\mathcal{D} (O)$ and the closed cone $V\subseteq \mathbb{R}^m$ satisfy $(\supp \varphi \times V) \cap L = \emptyset$. It is straightforward to verify that $\mathscr{H}^{s,r}_L(O)$ is complete. Furthermore, \cite[Proposition 3.13]{PP} implies that $\mathcal{D}(O)$ is sequentially dense in $\mathscr{H}_L ^{s,r} (O)$.

\begin{remark}
\label{Odnos_indeksa}
Since $H_{loc} ^s (O) \subseteq H_{loc} ^r (O)$ when $r\leq s$, $WF^{r} (u) = \emptyset$ for all $u\in H_{loc} ^s (O)$ and consequently $\mathscr{H}_L ^{s,r} (O)=H_{loc} ^s (O)$; it is straightforward to verify that the identity is topological.\\
\indent Furthermore, for any $s,r\in\mathbb{R}$, the following topological identities also hold true (cf. \cite[Remark 3.6]{PP}):
$$
\mathscr{H}_{\emptyset}^{s,r} (O)=H_{loc}^{\max\{s,r\}} (O),\quad \mathscr{H}_{O\times(\mathbb{R}^m\backslash\{0\})}^{s,r} (O)=H_{loc}^s(O).
$$
\end{remark}

\begin{proposition}
\label{Frese}
Let $s,r\in\mathbb{R}$ and let $L$ be a closed conic subset of $O\times(\mathbb{R}^m\backslash\{0\})$ satisfying $O\times(\mathbb{R}^m\backslash\{0\})\backslash L\neq \emptyset$.
Let the functions $\phi_{j,k}\in\mathcal{D}(O)$ and $\widetilde{\phi}_{j,k}\in\mathcal{C}^{\infty}(\mathbb{R}^m)$, $j,k\in\mathbb{N}$, be as in \cite[Proposition 3.7]{PP}. Then the mapping
\begin{align}
\mathscr{H}^{s,r}_L(O)\rightarrow H^s_{loc}(O)\times (L^2(\mathbb{R}^m)^{\mathbb{N}\times \mathbb{N}}),\quad u\mapsto(u,\mathbf{f}_u),\,\, \mbox{where}\label{equ-for-map}\\
\mathbf{f}_u(j,k):=\langle\cdot\rangle^r \widetilde{\phi}_{j,k}\mathcal{F}(\phi_{j,k} u),\, j,k\in\mathbb{N},\nonumber
\end{align}
is a well-defined topological imbedding with closed image. Consequently, $\mathscr{H}^{s,r}_L(O)$ is a reflexive Fr\'echet space.
\end{proposition}

\begin{proof}
The proof that the mapping is a well-defined continuous injection is straightforward and we omit it. Notice that \eqref{equ-for-map} is just the restriction of the map $\mathcal{I}:\mathcal{D}'^r_L(O)\rightarrow \mathcal{D}'(O)\times (L^2(\mathbb{R}^m)^{\mathbb{N}\times \mathbb{N}})$ in \cite[Proposition 3.7]{PP} to $\mathscr{H}^{s,r}_L(O)$. To show that \eqref{equ-for-map} is a topological imbedding, it suffices to prove that the seminorms $p_{r;\varphi,V}(u)$ and $q_{s;\psi}(u)$ of $u\in\mathscr{H}^{s,r}_L(O)$ are bounded by continuous seminorms on $H^s_{loc}(O)\times (L^2(\mathbb{R}^m)^{\mathbb{N}\times \mathbb{N}})$ of the image of $u$ under \eqref{equ-for-map}. This clearly holds for $q_{s;\psi}(u)$. It also holds for $p_{r;\varphi,V}(u)$ in view of \cite[Proposition 3.7]{PP} and the fact that the continuous seminorms on $\mathcal{D}'(O)$ of $u$ are bounded by continuous seminorms on $H^s_{loc}(O)$. This shows that \eqref{equ-for-map} is a topological imbedding. Its image is closed since $\mathscr{H}^{s,r}_L(O)$ is complete. The fact that $\mathscr{H}^{s,r}_L(O)$ is a reflexive Fr\'echet space now follows from the fact that the codomain of \eqref{equ-for-map} is a countable topological product of reflexive Fr\'echet spaces.
\end{proof}

\begin{remark}
Employing \cite[Remark 3.9]{PP}, we infer that the bilinear mapping $\mathcal{C}^{\infty}(O)\times \mathscr{H}^{s,r}_L(O)\rightarrow \mathscr{H}^{s,r}_L(O)$, $(\varphi ,u)\mapsto \varphi u$, is well-defined and separately continuous and, as both spaces are Fr\'echet, \cite[Theorem 1, p. 158]{Kete2} implies that it is continuous.
\end{remark}

Our next goal is to show that the following bilinear mapping is well-defined and continuous:
\begin{align}
\label{preslikavanje}
\mathscr{H}_{L_1} ^{r',r_1} (O) \times \mathscr{H}_{L_2} ^{r'',r_2} (U) \ni (u,v) \mapsto u\otimes v \in \mathscr{H}_{L} ^{s,r} (O\times U),
\end{align}
where
\begin{align}
\label{Uslovi_na_indeks1}
L &= (L_1 \times L_2) \cup (L_1 \times (U\times \{0\})) \cup ((O\times \{0\})\times L_2),\\
\label{Uslovi_na_indeks2}
r &\leq r_1 + \min \{0,r''\}, \quad r \leq r_2 + \min \{0,r'\},\\
\label{Uslovi_na_indeks3}
s &\leq r' + \min \{0,r'' \}, \quad s \leq r'' + \min \{0,r' \}.
\end{align}
The fact that $u\otimes v \in H_{loc} ^s (O\times U)$ follows from the following lemma.

\begin{lemma}
\label{LemaHlocs}
If $u\in H_{loc} ^{r'} (O) $ and $v\in H_{loc} ^{r''} (U) $ then $u\otimes v \in H_{loc} ^{s} (O \times U)$ for all $s\in \mathbb{R}$ such that $s \leq r' + \min \{0,r'' \},  s \leq r'' + \min \{0,r' \}$. Moreover, the bilinear mapping
$$H_{loc} ^{r'} (O) \times H_{loc} ^{r''} (U) \ni (u,v) \mapsto u\otimes v \in H_{loc} ^{s} (O\times U)$$
is continuous.
\end{lemma}
\begin{proof}
Let $s \leq r' + \min \{0,r'' \},  s \leq r'' + \min \{0,r' \}$ and $\chi\in\mathcal{D}(O\times U)$. Pick $\varphi \in \mathcal{D} (O)$ and $\psi\in \mathcal{D} (U)$ such that $\varphi\otimes \psi=1$ on a neighborhood of $\supp\chi$. Employing $\langle(\xi,\eta)\rangle^s\leq 2^{|s|}\langle (\xi-\xi',\eta-\eta')\rangle^s\langle(\xi',\eta')\rangle^{|s|}$ together with Young's inequality we infer
\begin{align}
q_{s;\chi}(u\otimes v)&=(2\pi)^{-n-m}\|\langle\cdot\rangle^s\mathcal{F}\chi*\mathcal{F}(\varphi u\otimes \psi v)\|_{L^2(\mathbb{R}^{n+m})}\nonumber\\
&\leq 2^{|s|}(2\pi)^{-n-m}\|\langle\cdot\rangle^{|s|}\mathcal{F}\chi\|_{L^1(\mathbb{R}^{n+m})}\|\langle\cdot\rangle^s\mathcal{F}(\varphi u\otimes \psi v)\|_{L^2(\mathbb{R}^{n+m})}.\label{bound-forone-t}
\end{align}
Consequently, we need to bound $q_{s; \varphi \otimes \psi} (u\otimes v)$. We discuss all possible signs of $r',r''$.
\begin{itemize}
\item \underline{$r',r''\geq 0$:}

Since $s \leq r' ,  s \leq r'' $, and $\langle (\xi , \eta) \rangle^s \leq \langle \xi \rangle^{r'} \langle \eta \rangle^{r''}$, we immediately deduce that
\begin{equation}\label{1_nejedn}
q_{s; \varphi \otimes \psi} ^2 (u\otimes v) 
= \iint_{\mathbb{R}^{m+n}} \langle (\xi , \eta) \rangle^{2s} | \widehat{\varphi u} (\xi) |^2 |\widehat{\psi v} (\eta)|^2 \mathrm{d}\xi \mathrm{d}\eta\leq q_{r'; \varphi } ^2 (u)  \ q_{r''; \psi} ^2 (v). 
\end{equation}

\item \underline{$r'<0 \leq r''$:}

This means that $s\leq r'$ and $s\leq r'+r''$, that is $s\leq r'<0$. Therefore, $\langle \xi \rangle^{|r'|} \leq \langle (\xi,\eta) \rangle^{|s|}$, and hence $\langle (\xi , \eta) \rangle^{s}\leq \langle \xi \rangle^{r'} \langle \eta \rangle^{r''}$. Consequently,
\begin{equation}\label{2_nejedn}
q_{s; \varphi \otimes \psi} ^2 (u\otimes v) 
\leq \iint_{\mathbb{R}^{2m}} \langle \xi \rangle^{2r'} \langle \eta \rangle^{2r''} | \widehat{\varphi u} (\xi) |^2 |\widehat{\psi v} (\eta)|^2 \mathrm{d}\xi \mathrm{d}\eta = q_{r'; \varphi } ^2 (u) \ q_{r''; \psi} ^2 (v).
\end{equation}

\item \underline{$r''<0 \leq r'$:}

This case is analogous to the previous one, that is
\begin{align}
\label{3_nejedn}
q_{s; \varphi \otimes \psi} ^2 (u\otimes v)
\leq q_{r'; \varphi } ^2 (u) \ q_{r''; \psi} ^2 (v)
\end{align}

\item \underline{$r',r''<0 $:}

In this case we have $s\leq r'+r''<0$, which implies
\begin{align*}
\langle (\xi , \eta) \rangle^{s}
\leq \langle (\xi , \eta) \rangle^{r'} \langle (\xi , \eta) \rangle^{r''}
\leq \langle \xi \rangle^{r'} \langle \eta \rangle^{r''}.
\end{align*}
Hence,
\begin{align}
\label{4_nejedn}
q_{s; \varphi \otimes \psi} ^2 (u\otimes v)
&\leq q_{r'; \varphi } ^2 (u) \ q_{r''; \psi} ^2 (v).
\end{align}
\end{itemize}
The continuity of the bilinear mapping follows from \eqref{bound-forone-t}, \eqref{1_nejedn}, \eqref{2_nejedn}, \eqref{3_nejedn}, \eqref{4_nejedn}.
\end{proof}

We can now show the continuity of \eqref{preslikavanje}.

\begin{theorem}
\label{hipo_tenz_proizv}
    Let $O\subseteq \mathbb{R}^m$ and $U\subseteq \mathbb{R}^n$ be open sets. The bilinear mapping \eqref{preslikavanje}, under the conditions \eqref{Uslovi_na_indeks1}, \eqref{Uslovi_na_indeks2}, \eqref{Uslovi_na_indeks3}, is well-defined and continuous.
\end{theorem}

\begin{proof}
When $L=O\times U\times(\mathbb{R}^{m+n}\backslash\{0\})$, the claim follows from Remark \ref{Odnos_indeksa} and Lemma \ref{LemaHlocs}. Assume now that $(O\times U\times(\mathbb{R}^{m+n}\backslash\{0\}))\backslash L\neq \emptyset$. Lemma \ref{LemaHlocs} together with Lemma \ref{WF_tenz_proiz} verify that \eqref{preslikavanje} is well-defined. To show that it is continuous, in view of Lemma \ref{LemaHlocs}, it remains to estimate $p_{r;\chi,V} (u\otimes v)$, $u\in \mathscr{H}_{L_1} ^{r',r_1} (O)$, $v\in\mathscr{H}_{L_2} ^{r'',r_2} (U)$, where $\chi\in\mathcal{D}(O\times U)\backslash\{0\}$ and the closed cone $V\subseteq \mathbb{R}^{m+n}$ satisfy $V\backslash\{0\}\neq \emptyset$ and $(\supp\chi\times V)\cap L=\emptyset$. Employing a standard compactness argument, we can find a compact set $K\subseteq O\times U$ and a closed cone $\widetilde{V}\subseteq \mathbb{R}^{m+n}$ such that $\supp\chi\subseteq \operatorname{int}K$, $V\backslash\{0\}\subseteq \operatorname{int}\widetilde{V}$ and $(K\times \widetilde{V})\cap L=\emptyset$. By employing the same technique as in the proof of \cite[Lemma 4.3]{Brouder}, one can show the following\\
\\
\noindent\textbf{Claim} There are open sets $O_j\subseteq O'_j\subseteq O$ and $U_j\subseteq U'_j\subseteq U$, $j=1,\ldots,l$, and closed cones $W_{j,1}\subseteq \mathbb{R}^m\backslash\{0\}$ and $W_{j,2}\subseteq\mathbb{R}^n\backslash\{0\}$, $j=1,\ldots,l$, which satisfy the following properties
\begin{itemize}
    \item[$(i)$] $\overline{O'_j}\times \overline{U'_j}\subseteq \operatorname{int}K$, $\overline{O_j}\subseteq O'_j$, $\overline{U_j}\subseteq U'_j$ and $\supp\chi\subseteq \bigcup_{j=1}^l(O_j\times U_j)$;
    \item[$(ii)$] $(\overline{O'_j}\times (\mathbb{R}^m\backslash W_{j,1}))\cap L_1=\emptyset$ and $(\overline{U'_j}\times(\mathbb{R}^n\backslash W_{j,2}))\cap L_2=\emptyset$;
    \item[$(iii)$] $\left((W_{j,1}\cup\{0\})\times (W_{j,2}\cup\{0\})\right)\cap (\widetilde{V}\backslash\{0\})=\emptyset$.
\end{itemize}
\vspace{10 pt}
\noindent Of course, $(\overline{O'_j}\times \overline{U'_j}\times \widetilde{V})\cap L=\emptyset$, $j=1,\ldots,l$. Pick nonnegative $\varphi_{j,1}\in\mathcal{D}(O'_j)$ and $\varphi_{j,2}\in\mathcal{D}(U'_j)$ such that $\varphi_{j,1}=1$ on $O_j$ and $\varphi_{j,2}=1$ on $U_j$. Set $\phi:=\sum_{j=1}^l \varphi_{j,1}\otimes\varphi_{j,2}$. Notice that
\begin{align*}
p_{r; \chi , V} (u\otimes v)&=\left(\iint_V \langle(\xi,\eta)\rangle^{2r}\left|\mathcal{F}\left((\chi/\phi)\sum_{j=1}^l\varphi_{j,1}u\otimes \varphi_{j,2}v\right)(\xi,\eta)\right|^2\mathrm{d}\xi \mathrm{d}\eta\right)^{\frac{1}{2}}\leq \sum_{j=1}^l I_j,
\end{align*}
where
\begin{equation*}
I_j:=\left(\iint_V \langle(\xi,\eta)\rangle^{2r}\left|\mathcal{F}\left((\chi/\phi)(\varphi_{j,1}u\otimes\varphi_{j,2}v)\right)(\xi,\eta)\right|^2\mathrm{d}\xi \mathrm{d}\eta\right)^{\frac{1}{2}},\quad j=1,\ldots, l.
\end{equation*}
There is $0<c<1$ such that
\begin{equation}\label{sub-for-cha-offc}
\{(\xi,\eta)\in\mathbb{R}^{m+n}\,|\, \exists(\xi',\eta')\in V\,\, \mbox{such that}\,\,|(\xi,\eta)-(\xi',\eta')|< c|(\xi',\eta')|\}\subseteq\operatorname{int} \widetilde{V}.
\end{equation}
We employ Minkowski integral inequality to infer
\begin{align*}
I_j&\leq \left(\iint_V \langle(\xi,\eta)\rangle^{2r}\left(\iint_{\mathbb{R}^{m+n}}|\mathcal{F}(\chi/\phi)(\xi',\eta')||\mathcal{F}(\varphi_{j,1}u\otimes\varphi_{j,2}v)(\xi-\xi',\eta-\eta')|\mathrm{d}\xi' \mathrm{d}\eta'\right)^2\mathrm{d}\xi \mathrm{d}\eta\right)^{\frac{1}{2}}\\
&\leq I'_j+I''_j,
\end{align*}
where
\begin{align*}
I'_j&:=2^{|r|}\iint_{\mathbb{R}^{m+n}} \langle(\xi',\eta')\rangle^{|r|}|\mathcal{F}(\chi/\phi)(\xi',\eta')|\\
&\cdot\left(\iint_{\substack{(\xi,\eta)\in V\\ |(\xi,\eta)|> |(\xi',\eta')|/c}}|\mathcal{F}(\varphi_{j,1}u\otimes\varphi_{j,2}v)(\xi-\xi',\eta-\eta')|^2\langle(\xi-\xi',\eta-\eta')\rangle^{2r}\mathrm{d}\xi \mathrm{d}\eta\right)^{\frac{1}{2}}\mathrm{d}\xi' \mathrm{d}\eta',\\
I''_j&:=2^{|r|}\iint_{\mathbb{R}^{m+n}} \langle(\xi',\eta')\rangle^{|r|}|\mathcal{F}(\chi/\phi)(\xi',\eta')|\\
&\cdot\left(\iint_{\substack{(\xi,\eta)\in V\\ |(\xi,\eta)|\leq |(\xi',\eta')|/c}}|\mathcal{F}(\varphi_{j,1}u\otimes\varphi_{j,2}v)(\xi-\xi',\eta-\eta')|^2\langle(\xi-\xi',\eta-\eta')\rangle^{2r}\mathrm{d}\xi \mathrm{d}\eta\right)^{\frac{1}{2}}\mathrm{d}\xi' \mathrm{d}\eta'.
\end{align*}
A change of variables in $I'_j$ together with \eqref{sub-for-cha-offc} gives
\begin{align*}
I'_j&=2^{|r|}\iint_{\mathbb{R}^{m+n}} \langle(\xi',\eta')\rangle^{|r|}|\mathcal{F}(\chi/\phi)(\xi',\eta')|\\
&{}\quad\cdot\left(\iint_{\substack{(\xi,\eta)\in V-(\xi',\eta')\\ |(\xi+\xi',\eta+\eta')|\geq |(\xi',\eta')|/c}}|\mathcal{F}(\varphi_{j,1}u\otimes\varphi_{j,2}v)(\xi,\eta)|^2\langle(\xi,\eta)\rangle^{2r}\mathrm{d}\xi \mathrm{d}\eta\right)^{\frac{1}{2}}\mathrm{d}\xi' \mathrm{d}\eta'\\
&\leq 2^{|r|}\|\langle\cdot\rangle^{|r|}\mathcal{F}(\chi/\phi)\|_{L^1(\mathbb{R}^{m+n})}p_{r;\varphi_{j,1}\otimes\varphi_{j,2},\widetilde{V}}(u\otimes v).
\end{align*}
To estimate $I''_j$, notice that on the domain of integration of the inner integral the following inequality holds true
\begin{equation*}
\langle(\xi-\xi',\eta-\eta')\rangle^{2r}\leq C''\langle(\xi-\xi',\eta-\eta')\rangle^{2s}\langle(\xi',\eta')\rangle^{4|r|+4|s|}.
\end{equation*}
Consequently, $I''_j\leq 2^{|r|}\sqrt{C''}\|\langle\cdot\rangle^{3|r|+2|s|}\mathcal{F}(\chi/\phi)\|_{L^1(\mathbb{R}^{m+n})}q_{s;\varphi_{j,1}\otimes\varphi_{j,2}}(u\otimes v)$ and we deduce
\begin{equation*}
p_{r; \chi , V} (u\otimes v)\leq C \sum_{j=1}^l(p_{r;\varphi_{j,1}\otimes\varphi_{j,2},\widetilde{V}}(u\otimes v)+q_{s;\varphi_{j,1}\otimes\varphi_{j,2}}(u\otimes v)).
\end{equation*}
Lemma \ref{LemaHlocs} implies that $q_{s;\varphi_{j,1}\otimes\varphi_{j,2}}(u\otimes v)$ is bounded by a product of a continuous seminorm on $H_{loc} ^{r'} (O)$ of $u$ and a continuous seminorm on $H_{loc} ^{r''} (U)$ of $v$. It remains to bound the seminorms $p_{r;\varphi_{j,1}\otimes\varphi_{j,2},\widetilde{V}}(u\otimes v)$, $j=1,\ldots,l$.\\
\indent For each $j=1,\ldots,l$, we claim that there exist closed cones $W'_{j,1}\subseteq \mathbb{R}^m\backslash\{0\}$ and $W'_{j,2}\subseteq \mathbb{R}^n\backslash\{0\}$ such that $W_{j,1}\subseteq \operatorname{int}W'_{j,1}$ and $W_{j,2}\subseteq \operatorname{int}W'_{j,2}$ and they satisfy
\begin{itemize}
\item[$(ii)'$] $(\overline{O'_j}\times (\mathbb{R}^m\backslash W'_{j,1}))\cap L_1=\emptyset$ and $(\overline{U'_j}\times(\mathbb{R}^n\backslash W_{j',2}))\cap L_2=\emptyset$;
\item[$(iii)'$] $((W'_{j,1}\cup\{0\})\times(W'_{j,2}\cup\{0\}))\cap(\widetilde{V}\backslash\{0\})=\emptyset$.
\end{itemize}
Notice that $(ii)'$ immediately follows from the property $(ii)$ in the Claim. To show $(iii)'$ we proceed as follows. For $k\in\mathbb{N}$ and $\lambda>0$, denote by $B_k(\lambda)$ the open ball in $\mathbb{R}^k$ with center at the origin and radius $\lambda$. There is $0<\varepsilon<1$ such that
\begin{equation}
\left(((W_{j,1}\cup\{0\})\times(W_{j,2}\cup\{0\}))\cap\mathbb{S}^{m+n}+\overline{B_{m+n}(\varepsilon)}\right)\cap\widetilde{V}=\emptyset.\label{conditions-on-setcfortk}
\end{equation}
We claim that
\begin{equation*}
W'_{j,1}:=\bigcup_{\lambda>0}\lambda\left(W_{j,1}\cap\mathbb{S}^m+\overline{B_m(\varepsilon)}\right)\quad \mbox{and}\quad W'_{j,2}:=\bigcup_{\lambda>0}\lambda\left(W_{j,2}\cap\mathbb{S}^n+\overline{B_n(\varepsilon)}\right)
\end{equation*}
satisfy $(iii)'$. Clearly $W'_{j,1}$ and $W'_{j,2}$ are closed cones in $\mathbb{R}^m\backslash\{0\}$ and $\mathbb{R}^n\backslash\{0\}$ respectively and satisfy $W_{j,1}\subseteq \operatorname{int}W'_{j,1}$ and $W_{j,2}\subseteq \operatorname{int}W'_{j,2}$. Assume that there is $(\xi_0,\eta_0)$ in the set in $(iii)'$. If $\xi_0=0$ then $\eta_0\in W'_{j,2}$ and consequently $\eta_0=\lambda(\eta+\eta')$ for some $\lambda>0$, $\eta\in W_{j,2}\cap\mathbb{S}^n$ and $\eta'\in\overline{B_n(\varepsilon)}$. But then $(0,\eta)+(0,\eta')=\lambda^{-1}(0,\eta_0)\in \widetilde{V}$ and $(0,\eta)+(0,\eta')\in ((W_{j,1}\cup\{0\})\times(W_{j,2}\cup\{0\}))\cap\mathbb{S}^{m+n}+\overline{B_{m+n}(\varepsilon)}$ which contradicts \eqref{conditions-on-setcfortk}. The case $\eta_0=0$ can be treated analogously. Assume now that both $\xi_0\neq0$ and $\eta_0\neq0$. Hence $\xi_0\in W'_{j,1}$ and $\eta_0\in W'_{j,2}$, and thus $\xi_0=\lambda(\xi+\xi')$ and $\eta_0=\mu(\eta+\eta')$ for some $\lambda,\mu>0$, $\xi\in W_{j,1}\cap\mathbb{S}^m$, $\eta\in W_{j,2}\cap\mathbb{S}^n$, $\xi'\in\overline{B_m(\varepsilon)}$ and $\eta'\in\overline{B_n(\varepsilon)}$. Notice that
\begin{equation*}
(1+\mu^2/\lambda^2)^{-1/2}(\xi,(\mu/\lambda)\eta)+(1+\mu^2/\lambda^2)^{-1/2}(\xi',(\mu/\lambda)\eta')=\lambda^{-1}(1+\mu^2/\lambda^2)^{-1/2}(\xi_0,\eta_0)\in\widetilde{V}.
\end{equation*}
This contradicts \eqref{conditions-on-setcfortk} since
\begin{gather*}
(1+\mu^2/\lambda^2)^{-1/2}(\xi,(\mu/\lambda)\eta)\in((W_{j,1}\cup\{0\})\times(W_{j,2}\cup\{0\}))\cap\mathbb{S}^{m+n}\quad\mbox{and}\\
(1+\mu^2/\lambda^2)^{-1/2}|(\xi',(\mu/\lambda)\eta')|\leq \varepsilon
\end{gather*}
and the proof of $(iii)'$ is complete. The properties in the Claim together with $(ii)'$ and $(iii)'$ imply that, for each $j=1,\ldots,l$, there are smooth nonnegative positively homogeneous functions of order zero $\alpha_j$ and $\beta_j$ on $\mathbb{R}^m\backslash\{0\}$ and $\mathbb{R}^n\backslash\{0\}$ respectively such that both are bounded by $1$ and satisfy the following properties:
\begin{enumerate}
\item
$\alpha_j|_{W_{j,1}} = 1, \ \beta_j|_{W_{j,2}} = 1$,
\item
$((\supp \alpha_j\cup\{0\}) \times (\supp \beta_j\cup\{0\})) \cap (\widetilde{V}\backslash\{0\}) = \emptyset$,
\item
$(\supp\varphi_{j,1}\times \supp (1-\alpha_j)) \cap L_1 = \emptyset$,
\item
$(\supp\varphi_{j,2} \times \supp (1-\beta_j)) \cap L_2 = \emptyset$.
\end{enumerate}
This gives
\begin{align*}
\widehat{\varphi_{j,1} u} (\xi) \widehat{\varphi_{j,2} v} (\eta) &= \alpha_j (\xi) \widehat{\varphi_{j,1} u} (\xi) \, \beta_j (\eta) \widehat{\varphi_{j,2} v} (\eta)\\
& \quad + \alpha_j (\xi) \widehat{\varphi_{j,1} u} (\xi) \, (1 - \beta_j (\eta)) \widehat{\varphi_{j,2} v} (\eta)\\
& \quad + ( 1 - \alpha_j (\xi))  \widehat{\varphi_{j,1} u} (\xi) \, \beta_j (\eta) \widehat{\varphi_{j,2} v} (\eta) \\
& \quad +  ( 1 - \alpha_j (\xi))  \widehat{\varphi_{j,1} u} (\xi) \, (1 - \beta_j (\eta)) \widehat{\varphi_{j,2} v} (\eta),
\end{align*}
and therefore
\begin{align*}
p_{r; \varphi_{j,1} \otimes \varphi_{j,2} , \widetilde{V}} (u \otimes v)
&\leq I_{j,1} + I_{j,2} + I_{j,3} + I_{j,4},
\end{align*}
where
\begin{align*}
I_{j,1}
&= \left( \iint_{\widetilde{V}} \langle (\xi , \eta) \rangle^{2r} \, |\alpha_j (\xi) \widehat{\varphi_{j,1} u} (\xi) |^2 \, |\beta_j (\eta) \widehat{\varphi_{j,2} v} (\eta)|^2 \mathrm{d}\xi \mathrm{d}\eta \right)^{\frac{1}{2}} ,\\
I_{j,2}
&= \left( \iint_{\widetilde{V}} \langle (\xi , \eta) \rangle^{2r} \, |\alpha_j (\xi) \widehat{\varphi_{j,1} u} (\xi) |^2 \, |(1-\beta_j (\eta)) \widehat{\varphi_{j,2} v} (\eta)|^2 \mathrm{d}\xi \mathrm{d}\eta \right)^{\frac{1}{2}} ,\\
I_{j,3}
&= \left( \iint_{\widetilde{V}} \langle (\xi , \eta) \rangle^{2r} |(1-\alpha_j (\xi)) \widehat{\varphi_{j,1} u} (\xi) |^2 |\beta_j (\eta) \widehat{\varphi_{j,2} v} (\eta)|^2 \, \mathrm{d}\xi \mathrm{d}\eta \right)^{\frac{1}{2}} ,\\
I_{j,4}
&= \left( \iint_{\widetilde{V}} \langle (\xi , \eta) \rangle^{2r} \, |(1-\alpha_j (\xi)) \widehat{\varphi_{j,1} u} (\xi) |^2 \, |(1-\beta_j (\eta)) \widehat{\varphi_{j,2} v} (\eta)|^2 \mathrm{d}\xi \mathrm{d}\eta \right)^{\frac{1}{2}} .
\end{align*}
We estimate each term separately.
\begin{itemize}
\item[$\bullet$ \underline{$I_{j,1}$:}]
Since $(\supp \alpha_j \times \supp \beta_j) \cap \widetilde{V} = \emptyset$, we infer $I_{j,1} = 0$.
\item[$\bullet$ \underline{$I_{j,4}$:}]
Denote $C_{\alpha_j}:= \supp (1-\alpha_j) , C_{\beta_j}:= \supp (1-\beta_j)$; of course $C_{\alpha_j}$ and $C_{\beta_j}$ are closed cones in $\mathbb{R}^m\backslash\{0\}$ and $\mathbb{R}^n\backslash\{0\}$. Observe that
\begin{align*}
\widetilde{V}  \cap (C_{\alpha_j} \times C_{\beta_j})
& \subseteq C_{\alpha_j} \times C_{\beta_j} \\
&= \{ (\xi , \eta) \in C_{\alpha_j} \times C_{\beta_j} \mid |\xi| \leq |\eta| \} \cup \{ (\xi , \eta) \in C_{\alpha_j} \times C_{\beta_j} \mid |\xi| > |\eta| \}.
\end{align*}

Hence,
\begin{align*}
I_{j,4}
&= \left( \iint_{\widetilde{V}} \langle (\xi , \eta) \rangle^{2r} \, |(1-\alpha_j (\xi)) \widehat{\varphi_{j,1} u} (\xi) |^2 \, |(1-\beta_j (\eta)) \widehat{\varphi_{j,2} v} (\eta)|^2 \mathrm{d}\xi \mathrm{d}\eta \right)^{\frac{1}{2}} \\
&\leq \left( \iint_{\substack{ \xi \in C_{\alpha_j} \\ \eta \in C_{\beta_j} \\ |\xi| \leq |\eta| }} \langle (\xi , \eta) \rangle^{2r} \, |(1-\alpha_j (\xi)) \widehat{\varphi_{j,1} u} (\xi) |^2 \, |(1-\beta_j (\eta)) \widehat{\varphi_{j,2} v} (\eta)|^2 \mathrm{d}\xi \mathrm{d}\eta \right)^{\frac{1}{2}} \\
& \quad + \left( \iint_{\substack{ \xi \in C_{\alpha_j} \\ \eta \in C_{\beta_j} \\ |\xi| > |\eta| }} \langle (\xi , \eta) \rangle^{2r} \, |(1-\alpha_j (\xi)) \widehat{\varphi_{j,1} u} (\xi) |^2 \, |(1-\beta_j (\eta)) \widehat{\varphi_{j,2} v} (\eta)|^2 \mathrm{d}\xi \mathrm{d}\eta \right)^{\frac{1}{2}} \\
&\leq \left( \iint_{\substack{ \xi \in C_{\alpha_j} \\ \eta \in C_{\beta_j} \\ |\xi| \leq |\eta| }} \langle (\xi , \eta) \rangle^{2r} \, |\widehat{\varphi_{j,1} u} (\xi) |^2 \, | \widehat{\varphi_{j,2} v} (\eta)|^2 \mathrm{d}\xi \mathrm{d}\eta \right)^{\frac{1}{2}} \\
& \quad + \left( \iint_{\substack{ \xi \in C_{\alpha_j} \\ \eta \in C_{\beta_j} \\ |\xi| > |\eta| }} \langle (\xi , \eta) \rangle^{2r} \, | \widehat{\varphi_{j,1} u} (\xi) |^2 \, | \widehat{\varphi_{j,2} v} (\eta)|^2 \mathrm{d}\xi \mathrm{d}\eta \right)^{\frac{1}{2}}.
\end{align*}
Now we discuss all possible signs of $r',r''$.
\begin{itemize}
\item[a) \underline{$r',r''\geq 0$:}]
Then it holds $r\leq r_1 + \min \{0,r''\}=r_1 , r\leq r_2 + \min \{0,r'\}=r_2$. If $|\xi| \leq |\eta|$ then $\langle (\xi , \eta) \rangle^{r} \leq c_r \langle \eta \rangle^{r} \leq c_r \langle \eta \rangle^{r_2}$. On the other hand, if $|\xi| > |\eta|$ then $\langle (\xi , \eta) \rangle^{r} \leq c_r \langle \xi \rangle^{r_1}$. This gives us
\begin{align*}
I_{j,4}
& \leq c_r \left( \iint_{\substack{ \xi \in C_{\alpha_j} \\ \eta \in C_{\beta_j} \\ |\xi| \leq |\eta| }} \langle  \eta \rangle^{2r_2} \, |\widehat{\varphi_{j,1} u} (\xi) |^2 \, | \widehat{\varphi_{j,2} v} (\eta)|^2 \mathrm{d}\xi \mathrm{d}\eta \right)^{\frac{1}{2}} \\
& \quad + c_r \left( \iint_{\substack{ \xi \in C_{\alpha_j} \\ \eta \in C_{\beta_j} \\ |\xi| > |\eta| }} \langle \xi  \rangle^{2r_1} \, | \widehat{\varphi_{j,1} u} (\xi) |^2 \, | \widehat{\varphi_{j,2} v} (\eta)|^2 \mathrm{d}\xi \mathrm{d}\eta \right)^{\frac{1}{2}} \\
& \leq c_r \left( \lVert \widehat{\varphi_{j,1} u} \rVert_{L^2 (\mathbb{R}^m)} p_{r_2 ; \varphi_{j,2} , C_{\beta_j}} (v) + p_{r_1 ; \varphi_{j,1} , C_{\alpha_j}} (u) \lVert \widehat{\varphi_{j,2} v} \rVert_{L^2 (\mathbb{R}^n)} \right). 
\end{align*}
\item[b) \underline{$r'<0\leq r''$:}]
Then it holds $r\leq r_1$ and $r\leq r_2 + r'$. If $|\xi|\leq |\eta|$ then $\langle (\xi , \eta) \rangle^{r} \leq c_r \langle \eta \rangle^{r_2} \langle \xi \rangle^{r'}$. On the other hand, if $|\xi| > |\eta|$ then $\langle (\xi , \eta) \rangle^{r} \leq c_r \langle \xi \rangle^{r_1}$. This gives us
\begin{align*}
I_{j,4}
& \leq c_r \left( \iint_{\substack{ \xi \in C_{\alpha_j} \\ \eta \in C_{\beta_j} \\ |\xi| \leq |\eta| }} \langle  \xi \rangle^{2r'} \langle  \eta \rangle^{2r_2} \, |\widehat{\varphi_{j,1} u} (\xi) |^2 \, | \widehat{\varphi_{j,2} v} (\eta)|^2 \mathrm{d}\xi \mathrm{d}\eta \right)^{\frac{1}{2}} \\
& \quad + c_r \left( \iint_{\substack{ \xi \in C_{\alpha_j} \\ \eta \in C_{\beta_j} \\ |\xi| > |\eta| }} \langle \xi  \rangle^{2r_1} \, | \widehat{\varphi_{j,1} u} (\xi) |^2 \, | \widehat{\varphi_{j,2} v} (\eta)|^2 \mathrm{d}\xi \mathrm{d}\eta \right)^{\frac{1}{2}} \\
& \leq c_r (q_{r' ; \varphi_{j,1}} (u) p_{r_2;\varphi_{j,2} , C_{\beta_j}} (v) + p_{r_1 ; \varphi_{j,1} , C_{\alpha_j}} (u) \lVert \widehat{\varphi_{j,2} v} \rVert_{L^2 (\mathbb{R}^n)}) .
\end{align*}
\item[c) \underline{$r''<0\leq r'$:}]
This case is analogous to the previous one. Precisely, from $r\leq r_1+r''$ and $r\leq r_2$, it follows $\langle (\xi , \eta) \rangle^r \leq c_r \langle \eta \rangle^{r_2}$ when $| \xi| \leq |\eta|$, and $\langle (\xi , \eta) \rangle^r \leq c_r \langle \xi \rangle^{r_1} \langle \eta \rangle^{r''}$ for $| \xi| > |\eta|$. Therefore,
\begin{align*}
I_{j,4}
\leq c_r (\|\widehat{\varphi_{j,1} u}\|_{L^2(\mathbb{R}^m)}p_{r_2;\varphi_{j,2} , C_{\beta_j}} (v) +  p_{r_1;\varphi_{j,1},C_{\alpha_j}}(u)q_{r'' ; \varphi_{j,2}} (v)).
\end{align*}
\item[d) \underline{$r',r''<0$:}]
Then it holds $r\leq r_1 + r''$ and $r\leq r_2 + r'$. If $|\xi|\leq |\eta|$ then $\langle (\xi , \eta ) \rangle^{r} \leq c_r \langle \eta \rangle^{r_2} \langle \xi \rangle^{r'}$. On the other hand, if $|\xi| > |\eta|$ then $\langle (\xi , \eta ) \rangle^{r} \leq c_r \langle \xi \rangle^{r_1} \langle \eta \rangle^{r''}$.
\begin{align*}
I_{j,4}
& \leq c_r \left( \iint_{\substack{ \xi \in C_{\alpha_j} \\ \eta \in C_{\beta_j} \\ |\xi| \leq |\eta| }} \langle  \xi \rangle^{2r'} \langle  \eta \rangle^{2r_2} \, |\widehat{\varphi_{j,1} u} (\xi) |^2 \, | \widehat{\varphi_{j,2} v} (\eta)|^2 \mathrm{d}\xi \mathrm{d}\eta \right)^{\frac{1}{2}} \\
& \quad + c_r \left( \iint_{\substack{ \xi \in C_{\alpha_j} \\ \eta \in C_{\beta_j} \\ |\xi| > |\eta| }} \langle  \xi \rangle^{2r_1} \langle  \eta \rangle^{2r''} \, |\widehat{\varphi_{j,1} u} (\xi) |^2 \, | \widehat{\varphi_{j,2} v} (\eta)|^2 \mathrm{d}\xi \mathrm{d}\eta \right)^{\frac{1}{2}} \\
& \leq c_r (q_{r' ; \varphi_{j,1}} (u) p_{r_2; \varphi_{j,2} , C_{\beta_j}} (v) + p_{r_1; \varphi_{j,1} , C_{\alpha_j}} (u) q_{r'' ; \varphi_{j,2}} (v)) .
\end{align*}
\end{itemize}

\item[$\bullet$ \underline{$I_{j,2}$:}]
The domain of integration in $I_{j,2}$ is $(\supp \alpha_j \times C_{\beta_j})\cap\widetilde{V}$. Notice that for each $(\xi,\eta)\in((\supp\alpha_j\cup\{0\})\times(C_{\beta_j}\cup\{0\}))\cap\widetilde{V}\cap\mathbb{S}^{m+n}$ it holds that $\eta\neq 0$. Indeed, if $(\xi,0)$ belongs to this set for some $\xi$, then $(\xi,0)$ would belong to the set in 2. which is a contradiction. Since $((\supp\alpha_j\cup\{0\})\times(C_{\beta_j}\cup\{0\}))\cap\widetilde{V}\cap\mathbb{S}^{m+n}$ is compact, there is $0<\varepsilon'<1$ such that for every $(\xi,\eta)$ in this set it holds that $|\eta|\geq \varepsilon'$. This implies that if $(\xi,\eta)\in ((\supp\alpha_j\cup\{0\})\times(C_{\beta_j}\cup\{0\}))\cap\widetilde{V}$ then $|\xi|\leq \varepsilon'^{-1}\sqrt{1-\varepsilon'^2}|\eta|$.
\begin{itemize}
\item[a) \underline{$r' \geq 0$:}]
We have $\langle (\xi,\eta)\rangle^r\leq c'\langle \eta\rangle^{r_2}$ and hence
\begin{align*}
I_{j,2}
&\leq c' \left( \iint_{ \substack{ \xi \in \supp \alpha_j \\ \eta \in C_{\beta_j} } } \langle \eta \rangle^{2r_2} \, |\alpha_j (\xi) \widehat{\varphi_{j,1} u} (\xi) |^2 \, |(1-\beta_j (\eta)) \widehat{\varphi_{j,2} v} (\eta)|^2 \mathrm{d}\xi \mathrm{d}\eta \right)^{\frac{1}{2}} \\
&\leq c' \lVert \widehat{\varphi_{j,1} u} \rVert_{L^2 (\mathbb{R}^m)} p_{r_2 ; \varphi_{j,2} , C_{\beta_j}} (v) .
\end{align*}
\item[b) \underline{$r' < 0$:}]
This means that $r\leq r_2 + r'$ and therefore, $\langle (\xi , \eta) \rangle^{r} \leq c' \langle \eta \rangle^{r_2} \langle \xi \rangle^{r'}$. Hence,
\begin{align*}
I_{j,2}
&\leq c' \left( \iint_{ \substack{ \xi \in \supp \alpha_j \\ \eta \in C_{\beta_j} } } \langle \eta \rangle^{2r_2} \langle \xi \rangle^{2r'} \, |\alpha_j (\xi) \widehat{\varphi_{j,1} u} (\xi) |^2 \, |(1-\beta_j (\eta)) \widehat{\varphi_{j,2} v} (\eta)|^2 \mathrm{d}\xi \mathrm{d}\eta \right)^{\frac{1}{2}} \\
&\leq c' q_{r'; \varphi_{j,1}} (u) p_{r_2 ; \varphi_{j,2} , C_{\beta_j}} (v) .
\end{align*}
\end{itemize}
\item[$\bullet$ \underline{$I_{j,3}$:}]
Analogously as for $I_{j,2}$, one obtains $I_{j,3} \leq c  p_{r_1; \varphi_{j,1} , C_{\alpha_j}} (u) \lVert \widehat{\varphi_{j,2} v} \rVert_{L^2 (\mathbb{R}^{n})}$ when $r''\geq 0$ and $ I_{j,3} \leq c p_{r_1; \varphi_{j,1} , C_{\alpha_j}} (u) \, q_{r'' ; \varphi_{j,2}} (v)$ when $r'' < 0$.
\end{itemize}
The proof of the theorem is complete.
\end{proof}

\begin{remark}\label{rtksloh-kfvss}
The bilinear map $\mathcal{D}(O)\times \mathcal{D}(O)\rightarrow \mathcal{D}(O\times O)$, $(\psi,\phi)\mapsto \psi\phi$, uniquely extends to a continuous bilinear map $H_{loc} ^{r'} (O)\times H_{loc} ^{r''} (O)\rightarrow H_{loc} ^{s_*} (O)$, when $r'+r''\geq0$, $s_*\leq \min\{r',r''\}$ and $s_*\leq r'+r''-\frac{m}{2}$ with the last inequality being strict if $-s_*$ or $r'$ or $r''$ equals $m/2$. To verify this, let $u,v\in\mathcal{D}(O)$ and $\varphi\in\mathcal{D}(O)$. Pick $\varphi_1\in\mathcal{D}(O)$ such that $\varphi_1=1$ on a neighborhood of $\supp\varphi$. Repeating the same arguments verbatim as in the proof of \cite[Theorem 8.3.1, p. 189]{HNonlinear} (cf. \cite{beh-h,TJPT}) one can show that
\begin{equation*}
q_{s_*;\varphi}(uv)=\|\langle\cdot\rangle^{s_*}\mathcal{F}(\varphi u \varphi_1 v)\|_{L^2(\mathbb{R}^m)}\leq C \|\langle\cdot\rangle^{r'}\mathcal{F}(\varphi u)\|_{L^2(\mathbb{R}^m)} \|\langle\cdot\rangle^{r''}\mathcal{F}(\varphi_1 v)\|_{L^2(\mathbb{R}^m)}.
\end{equation*}
The claim now follows from the denseness of $\mathcal{D}(O)$ in $H_{loc} ^{r'} (O)$ and $H_{loc} ^{r''} (O)$.
\end{remark}

Our main result on the product of distributions in $\mathscr{H}_{L_1} ^{r',r_1} (O)$ and $\mathscr{H}_{L_2} ^{r'',r_2} (O)$ is the following theorem.

\begin{theorem}\label{main-the-prod-distk}
Let $O\subseteq \mathbb{R}^m$ be an open set. Let $r', r'',r_1,r_2\in \mathbb{R}$ and assume that
\begin{equation}\label{asumpton-ind-klsss}
r:=\min\{r_1+\min\{0,r''\},r_2+\min\{0,r'\}\}>m/2,\quad r'+r''\geq0.
\end{equation}
Let $L_1$ and $L_2$ be closed conic subsets of $O\times(\mathbb{R}^m\backslash\{0\})$ which satisfy
\begin{equation}
\label{uslovTe}
\text{there is no } (x,\xi)\in L_1 \text{ such that } (x,-\xi)\in L_2
\end{equation}
and define $L$ as in \eqref{Uslovi_na_indeks1}. The pullback $\delta^*:\mathcal{C}^{\infty}(O\times O)\rightarrow\mathcal{C}^{\infty}(O)$, $\delta^* f=f\circ \delta$, by the diagonal map $\delta:O\rightarrow O\times O$, $\delta(x)=(x,x)$, uniquely extends to a continuous map
\begin{equation}\label{pullback-by-diagonma}
\delta^*:\mathcal{D}'^r_L (O\times O)\rightarrow \mathcal{D}'^{r_*}_{\delta^* L}  (O)
\end{equation}
for any $r_*\in\mathbb{R}$ satisfying $r_*\leq r-\frac{m}{2}$. For any $u\in\mathscr{H}_{L_1} ^{r',r_1} (O)$ and $v\in\mathscr{H}_{L_2} ^{r'',r_2} (O)$, $u\otimes v\in \mathcal{D}'^r_L(O\times O)$ and $\delta^*(u\otimes v)\in\mathscr{H}_{\delta^* L} ^{s_*, \, r_*} (O)$ for any $s_*\in\mathbb{R}$ satisfying $s_*\leq \min\{r',r''\}$ and $s_*\leq r'+r''-\frac{m}{2}$ with the last inequality being strict if $-s_*$ or $r'$ or $r''$ equals $m/2$. Furthermore, the bilinear map
\begin{equation}\label{map-cont-pulba-mult}
\mathscr{H}_{L_1} ^{r',r_1} (O)\times\mathscr{H}_{L_2} ^{r'',r_2} (O)\rightarrow \mathscr{H}_{\delta^* L} ^{s_*, \, r_*} (O),\quad (u,v)\mapsto \delta^*(u \otimes v),
\end{equation}
is continuous and it restricts to the ordinary pointwise multiplication of smooth functions.
\end{theorem}

\begin{proof}
For all $x\in O$ it holds $ ^t \delta ' (x) = [I \ I]$, where $I$ denotes the identity $m\times m$ matrix. For $\xi, \eta \in \mathbb{R}^m$ we have $ ^t \delta ' (x) [\xi \ \eta]^t = [I \ I] \cdot [\xi \ \eta]^t = \xi+\eta$, which means
\begin{equation*}
\mathcal{N}_{\delta}= \{ (\delta (x) ; \xi , \eta) \mid \, ^t \delta ' (x) [\xi \ \eta]^t = 0 , x\in O \}= \{ (x,x ; \xi , -\xi) \mid x\in O , \xi \in \mathbb{R}^m \}.
\end{equation*}
If $(x,y;\xi,\eta)\in L \cap \mathcal{N}_{\delta}$, then $y=x,\eta=-\xi$. Due to the definition of $L$ and $\xi\neq 0$, we get $(x,\xi)\in L_1$ and $(x,-\xi)\in L_2$, which contradicts with the assumption \eqref{uslovTe}. Therefore $\mathcal{N}_{\delta} \cap L = \emptyset$. Since $\delta$ has constant rank $m$, Theorem \ref{PPteorema} implies that \eqref{pullback-by-diagonma} is well-defined and continuous when $r_*\leq r-\frac{m}{2}$ (as $r>\frac{m}{2}$ by assumption). For $u\in\mathscr{H}_{L_1} ^{r',r_1} (O)$ and $v\in\mathscr{H}_{L_2} ^{r'',r_2} (O)$, Theorem \ref{hipo_tenz_proizv} yields $u\otimes v\in \mathcal{D}'^r_L (O\times O)$ and together with the above it also shows that the bilinear mapping
\begin{equation}\label{map-pulbackcontionsks}
\mathscr{H}_{L_1} ^{r',r_1} (O)\times\mathscr{H}_{L_2} ^{r'',r_2} (O)\rightarrow \mathcal{D}'^{r_*}_{\delta^* L}(O),\quad (u,v)\mapsto \delta^*(u \otimes v),
\end{equation}
is well-defined and continuous. Of course, when $u$ and $v$ are smooth functions, $\delta^*(u\otimes v)=uv$. It remains to show that \eqref{map-cont-pulba-mult} is well-defined and continuous, i.e. that the codomain in \eqref{map-pulbackcontionsks} can be restricted to $\mathscr{H}_{\delta^* L} ^{s_*, \, r_*} (O)$ and the resulting map is continuous. In view of Remark \ref{rtksloh-kfvss}, $\mathcal{D}(O)\times \mathcal{D}(O)\rightarrow \mathcal{D}(O\times O)$, $(\psi,\phi)\mapsto \psi\phi$, uniquely extends to a continuous bilinear map $P:H_{loc} ^{r'} (O)\times H_{loc} ^{r''} (O)\rightarrow H_{loc} ^{s_*} (O)$. As it coincides with $(u,v)\mapsto \delta^*(u\otimes v)$ on the dense subspace $\mathcal{D}(O)$ of $\mathscr{H}_{L_1} ^{r',r_1} (O)$ and $\mathscr{H}_{L_2} ^{r'',r_2} (O)$, we infer that $P(u,v)=\delta^*(u\otimes v)$, $u\in \mathscr{H}_{L_1} ^{r',r_1} (O)$, $v\in \mathscr{H}_{L_2} ^{r'',r_2} (O)$. Hence \eqref{map-cont-pulba-mult} is well-defined and the above bounds together with the continuity of \eqref{map-pulbackcontionsks} verify that \eqref{map-cont-pulba-mult} is continuous, which completes the proof of the theorem.
\end{proof}

The theorem verifies that we can define the product of distributions in $\mathscr{H}_{L_1} ^{r',r_1} (O)$ and $\mathscr{H}_{L_2} ^{r'',r_2} (O)$ as the continuous bilinear mapping
\begin{equation}\label{map-cont-pulba-mult11}
\mathscr{H}_{L_1} ^{r',r_1} (O)\times\mathscr{H}_{L_2} ^{r'',r_2} (O)\rightarrow \mathscr{H}_{\delta^* L} ^{s_*, \, r_*} (O),\quad uv:= \delta^*(u \otimes v),
\end{equation}
where $L_1$, $L_2$, $L$ and $r', r'',r_1,r_2, r_*,s_*\in \mathbb{R}$ satisfy the assumptions in the theorem.\\
\indent Finally, we show that any other sensible definition of a product on $\mathscr{H}_{L_1} ^{r',r_1} (O)\times\mathscr{H}_{L_2} ^{r'',r_2} (O)$ that restricts to the ordinary product of smooth functions has to coincide with our definition.

\begin{proposition}\label{ext-mulgbsk}
Let $O\subseteq \mathbb{R}^m$ be an open set. Let $r', r'',r_1,r_2\in \mathbb{R}$ satisfy \eqref{asumpton-ind-klsss} and let the closed conic sets $L_1,L_2\subseteq O\times(\mathbb{R}^m\backslash\{0\})$ satisfy \eqref{uslovTe}. Let $X$ and $Y$ be spaces of distributions for which the following continuous and dense inclusions hold true:
\begin{equation*}
\mathcal{D}(O)\subseteq X\subseteq \mathscr{H}_{L_1} ^{r',r_1} (O)\quad \mbox{and}\quad \mathcal{D}(O)\subseteq Y\subseteq \mathscr{H}_{L_2} ^{r'',r_2} (O).
\end{equation*}
If $P:X\times Y\rightarrow \mathcal{D}'(O)$ is a separately continuous bilinear map which restricts to the ordinary pointwise product on $\mathcal{D}(O)\times\mathcal{D}(O)$ then $P(u,v)=uv$, $u\in X$, $v\in Y$, where the product on the right is the one defined above. In particular, the codomain can be restricted to $\mathscr{H}_{\delta^* L} ^{s_*, \, r_*} (O)$ with $s_*$ and $r_*$ as in Theorem \ref{main-the-prod-distk} and the map is continuous.
\end{proposition}

\begin{proof}
Let $v\in Y$ be arbitrary but fixed. The continuous map $X\mapsto \mathcal{D}'(O)$, $u\mapsto P(u,v)$, coincides with the continuous map $X\mapsto \mathcal{D}'(O)$, $u\mapsto uv$, on the dense subset $\mathcal{D}(O)$ of $X$ and consequently $P(u,v)=uv$, $u\in X$. The rest of the claims follow immediately from Theorem \ref{main-the-prod-distk}.
\end{proof}

\begin{remark}
In the claims of Theorem \ref{main-the-prod-distk}, the conditions on $r_*$ are, in general, more restrictive than the conditions on $s_*$. This suggests that a more refined variant of Theorem \ref{main-the-prod-distk} (providing better bounds on $r_*$) might hold true.
\end{remark}

With the help of Remark \ref{rtksloh-kfvss}, we can improve the part of Theorem \ref{main-the-prod-distk} concerning the bounds on the index $r_*$ in the special case when $L_j=F_j\times(\mathbb{R}^m\backslash\{0\})$, $j=1,2$, where $F_1$ and $F_2$ are  disjoint closed sets in $O$ (which further supports the above remark). In this case, notice that $\mathscr{H}_{L_1} ^{s,r} (O)=\{u\in H_{loc} ^s (O)\,|\, \singsupp_ru\subseteq F_1\}$, where $\singsupp_ru$ is the Sobolev singular support of order $r$ of $u$, i.e.
\begin{align}
\singsupp_{r} u := \{ x\in O \mid \neg \, (\exists O_x \text{ an open neighborhood of } x) \, u\in H_{loc} ^{r} (O_x)  \}
\end{align}
(recall that $pr_1 (WF^r (u))=\singsupp_{r} u$, where $pr_1(\cdot)$ stands for the projection on the first variable).

\begin{proposition}
Let $F_1$ and $F_2$ be disjoint closed sets in the open set $O\subseteq \mathbb{R}^m$ and set $L_j:=F_j\times(\mathbb{R}^m\backslash\{0\})$, $j=1,2$, $L_0:=L_1\cup L_2$. Let $r', r'',r_1,r_2\in \mathbb{R}$ and assume that
\begin{equation*}
r_1+r_2\geq 0\quad \mbox{and}\quad r'+r''\geq0.
\end{equation*}
With $\delta$ as in Theorem \ref{main-the-prod-distk}, the pullback $\delta^*:\mathcal{C}^{\infty}(O\times O)\rightarrow\mathcal{C}^{\infty}(O)$, $\delta^* f=f\circ \delta$, uniquely extends to the continuous bilinear map\footnote{With $L$ as in \eqref{Uslovi_na_indeks1}, notice that $\delta^*L=L_0$.}
\begin{equation*}
\mathscr{H}_{L_1} ^{r',r_1} (O)\times\mathscr{H}_{L_2} ^{r'',r_2} (O)\rightarrow \mathscr{H}_{L_0} ^{s_*, \, r_*} (O),\quad (u,v)\mapsto \delta^*(u \otimes v),
\end{equation*}
for any $r_*,s_*\in\mathbb{R}$ satisfying
\begin{gather}
s_*\leq \min\{r',r''\},\quad s_*\leq r'+r''-\frac{m}{2},\label{inshk-rtvk}\\
r_*\leq \min\{r_1,r_2\},\quad r_*\leq r_1+r_2-\frac{m}{2},\label{inq-st-for-klkst}
\end{gather}
with the last inequality in \eqref{inshk-rtvk} being strict if $-s_*$ or $r'$ or $r''$ equals $m/2$ and the last inequality in \eqref{inq-st-for-klkst} being strict if $-r_*$ or $r_1$ or $r_2$ equals $m/2$.
\end{proposition}

\begin{proof}
Let $O_1,O_2\subseteq O$ be disjoint open sets which satisfy $F_1\subseteq O_1$ and $F_2\subseteq O_2$. Pick $\psi_1,\psi_2\in\mathcal{C}^{\infty}(O)$ such that $\psi_j=1$ on a neighborhood of $F_j$ and $\supp\psi_j\subseteq O_j$, $j=1,2$; consequently $\supp\psi_1\cap\supp\psi_2=\emptyset$. For $u,v\in \mathcal{D}(O)$, write
\begin{equation*}
uv= \underbrace{(\psi_1 u )(\psi_2 v )}_{=0}
	+ 
	(\psi_1 u )	((1 - \psi_2) v)
	+ 
	((1 - \psi_1) u)(\psi_2 v)+ ((1 - \psi_1) u)((1 - \psi_2) v).
\end{equation*}
Let $\varphi\in\mathcal{D}(O)$ satisfy $\supp\varphi\cap(F_1\cup F_2)=\emptyset$ (i.e. $(\supp\varphi\times \mathbb{R}^m)\cap L_0=\emptyset$) and pick $\phi\in\mathcal{D}(O)$ such that $\phi=1$ on a neighborhood of $\supp\varphi$ and $\supp\phi\cap(F_1\cup F_2)=\emptyset$. The assumptions on $r_*$ together with Remark \ref{rtksloh-kfvss} implies that there are $\varphi_j\in\mathcal{D}(O)$, $j=1,\ldots,6$, and $C>0$ such that
\begin{align*}
q_{r_*;\varphi}((\psi_1 u )	((1 - \psi_2) v))&=q_{r_*;\varphi}((\psi_1\phi u )	((1 - \psi_2) v))\leq Cq_{r_1;\varphi_1}(\psi_1\phi u ) q_{r_2;\varphi_2}((1 - \psi_2) v)\\
&=Cq_{r_1;\varphi_1\psi_1\phi}(u) q_{r_2;\varphi_2(1 - \psi_2)}(v),\\
q_{r_*;\varphi}(((1-\psi_1) u )(\psi_2 v))&=q_{r_*;\varphi}(((1-\psi_1) u )(\psi_2\phi v))\leq Cq_{r_1;\varphi_3}((1-\psi_1) u ) q_{r_2;\varphi_4}(\psi_2\phi v)\\
&=Cq_{r_1;\varphi_3(1-\psi_1)}(u) q_{r_2;\varphi_4\psi_2\phi }(v),\\
q_{r_*;\varphi}(((1 - \psi_1) u)((1 - \psi_2) v))&\leq Cq_{r_1;\varphi_5}((1 - \psi_1) u)q_{r_2;\varphi_6}((1 - \psi_2) v)\\
&= Cq_{r_1;\varphi_5(1 - \psi_1)}(u)q_{r_2;\varphi_6(1 - \psi_2)}(v).
\end{align*}
In view of these inequalities and Remark \ref{rtksloh-kfvss}, we deduce that $\mathcal{D}(O)\times \mathcal{D}(O)\rightarrow \mathcal{D}(O\times O)$, $(u,v)\mapsto uv$, uniquely extends to a continuous bilinear map $\mathscr{H}_{L_1} ^{r',r_1} (O)\times\mathscr{H}_{L_2} ^{r'',r_2} (O)\rightarrow \mathscr{H}_{L_0} ^{s_*, \, r_*} (O)$. Proposition \ref{ext-mulgbsk} implies that this map is given by $(u,v)\mapsto \delta^*(u\otimes v)$.
\end{proof}

\section*{Acknowledgments}
This research was  supported by the Science Fund of the Republic of Serbia, $\#$GRANT No 2727, \emph{Global and local analysis of operators and distributions} - GOALS. The author also acknowledges the support of the Ministry of Science, Technological Development and Innovation of the Republic of Serbia (Grants No. 451-03-137/2025-03/ 200125 $\&$ 451-03-136/2025-03/ 200125).

\end{document}